\documentclass{amsart}

\newtheorem{theorem}{Theorem}
\theoremstyle{plain}
\newtheorem{corollary}{Corollary}
\newtheorem{lemma}{Lemma}

\usepackage{bbold}
\usepackage{hyperref}

\newcommand{\RR}{\mathbb R}
\newcommand{\eins}{\mathbb 1}

\newcommand{\Sn}{S_n}

\newcommand{\set}[1]{\left\{\,#1\,\right\}}  % Menge
\newcommand{\with}{\ \vrule\ }  % 'mit'-Symbol in Mengen

\newcommand{\defa}{:=} %definiert als

\newcommand{\tk}{\tilde{k}}
\newcommand{\tw}{\tilde{w}}
\newcommand{\kg}{{\mathbb R}^{\Sn}}
\newcommand{\Uinv}{U_{Inv}}
\newcommand{\scp}[2]{\langle #1,#2 \rangle}
\newcommand{\poly}[1]{P(#1)}
\newcommand{\polyy}[1]{\mathfrak{P}(#1)}

\newcommand{\LOP}[1][n]{P_{#1}}
\newcommand{\lop}{p}
\newcommand{\wn}{\omega_0}
\newcommand{\cyc}{\mathfrak{c}}
\begin{document}

\begin{abstract}
Let $P_n$ denote the $n$-th linear ordering polytope. We define projections from $P_n$ to the $n$-th permutahedron and to the $(n-1)$-st linear ordering polytope. Both projections are equivariant with respect to the natural $\Sn$-action and they project to orthogonal subspaces. In particular the second projection defines an $S_n$-action in $P_{n-1}$.
%This gives an interpretation of the automorphism group derived in \cite{fiorini2001determining}.
\end{abstract}
\title{The Linear Ordering Polytope via Representations}
\author{Lukas Katth\"an}
\address{Fachbereich Mathematik und Informatik, Philipps Universit\"at, 35032 Marburg, Germany}%
\email{katthaen@mathematik.uni-marburg.de}%

 %, 05E45; Secondary 13F55}

\maketitle
\section{Introduction}
Let $\Sn$ denote the symmetric group. For a permutation $\pi \in\Sn$ and $1\leq i<j\leq n$ we set 
\[k_{ij}(\pi) \defa 
\begin{cases}
1 &\textnormal{ if } \pi(i) > \pi(j) \\
0 &\textnormal{ if } \pi(i) < \pi(j) \\
\end{cases}\]
The $n$-th \emph{Linear Ordering Polytope} $\LOP$ is defined as the convex hull of the $n!$ vectors $\lop_\pi \defa (k_{ij}(\pi))_{1\leq i<j\leq n} \in \RR^{\binom{n}{2}}$.
In this we follow the definition given in \cite{fishburn1992induced}.
The linear ordering polytope is an important and well-studied object in combinatorial optimization, see for example Chapter $6$ of \cite{Reinelt2011} and the references therein. For its study we will also consider the $n$-th \emph{permutahedron}; this is the polytope with $n!$ vertices $(\pi(i))_{1\leq i \leq n} \in \RR^n$ for $\pi \in S_n$.
Both polytopes carry a natural $S_n$-action and a $\mathbb{Z}_2$-action, which we will explain in Section 3.

This is the main result of this note: % In this note, we present a decomposition as follows:
\begin{theorem}\label{thm:main}
There is a scalar product on $\RR^{\binom{n}{2}}$ such that there is an orthogonal decomposition $\RR^{\binom{n}{2}} = V_1 \bot V_2$ into subspaces of dimensions $n-1$ and $\binom{n-1}{2}$, such that:
\begin{enumerate}
\item The orthogonal projection of $\LOP$ onto $V_1$ is the $n$-th permutahedron.
\item The orthogonal projection onto $V_2$ is $\LOP[n-1]$. 
\item Both projections are equivariant with respect to the $\mathbb{Z}_2 \times S_{n}$-action on $\LOP$. This gives rise to a $\mathbb{Z}_2 \times S_{n}$-action on $\LOP[n-1]$.
\end{enumerate}
\end{theorem}
If we apply the theorem repeatedly, we get 
\begin{corollary}
There is a scalar product on $\RR^{\binom{n}{2}}$ such that there is an orthogonal decomposition $\RR^{\binom{n}{2}} = V_n \bot V_{n-1} \bot \ldots \bot V_2$ satisfying
\begin{enumerate}
\item $\dim V_i = i-1$.
\item The orthogonal projection of $\LOP$ onto $V_i$ is the $i$-th permutahedron.
\end{enumerate}
\end{corollary}

The pure existence of the vector space decomposition $V_1 \bot V_2$ in Theorem \ref{thm:main} was known before.
First, it seems to have been known to experts in combinatorial optimization.
There, $\RR^{\binom{n}{2}}$ is considered as the space of functions on the edges of a complete directed graph. Then $V_1$ is the subspace of potentials (all functions $f$ such that $f(i,j) = g(i) - g(j), 1\leq i<j\leq n$ for some function $g$ defined on the nodes). Moreover $V_2$ is the space of circulations (all functions $f$ such that $\sum_i f(i,j) = \sum_k f(j,k)$ for all $j$). We thank S. Fiorini for bringing this fact to our attention.

Second, these spaces arise in the context of the analysis of rank data, see \cite{Marden}. There the $\RR^{\binom{n}{2}}$ is considered as the \emph{space of pairs} and $V_1$ is the \emph{space of means}. They correspond to the \emph{Babington Smith Model} resp. the \emph{Bradley/TerryMallows model}. In Section $7.4.3$ in \cite{Marden}, there are also a basis of $V_1$ and an generating set of $V_2$ given. The later contains the basis we give in Section 3. 

Third, the decomposition arises from representation theoretic considerations in \cite{2011arXiv1102.2460R}. In that paper it is also shown to be unique.

We also note that the existence of the $\mathbb{Z}_2 \times S_{n}$-action on $\LOP[n-1]$ is known from \cite{fiorini2001determining}, where this action is constructed from the $\mathbb{Z}_2 \times S_{n-1}$-action and a certain class of automorphisms borrowed from \cite{Bolotashvili}.

The projection $\LOP \rightarrow \LOP[n-1]$ maps $n$ vertices to one.
In \cite[Lemma 2]{fiorini2001determining} the \emph{trivial} and \emph{$3$-cycle} facets are shown to be the only facets having the maximal number of $n!/2$ vertices. Thus the preimages of these facets under the projection are also facets of the same type. Since facets inequalities for $\LOP$ are a major research topic, it might be interesting to consider the other known families of facets (see for example \cite{Reinelt2011}) from this point of view.

The construction of the linear ordering polytope can be generalized to arbitrary finite Coxeter groups. See \cite{Fiorini} for Type $B$ and \cite[Section 3.5]{2011arXiv1102.2460R} for the general case. There is still a decomposition into two invariant subspaces as in Theorem \ref{thm:main}. The space $V_1$ generalizes as expected: The action of the group is the geometric representation, so it is still irreducible and the polytope can be described as the dual of the Coxeter complex.
We know less about the other space $V_2$: We do not know if the representation is still irreducible and we also do not know the polytope. 
One could expect the polytope to be the polytope associated to the Coxeter group where one generator from one end of the Dynkin diagram was removed.  This would generalize the $(n-1)$st linear ordering polytope encountered in Theorem \ref{thm:main}. But this fails for dimensional reasons.
The representation of the group on this space is isomorphic to the one called $\psi_c$ in \cite[Theorem 37]{Renteln}. In the prove of Theorem 36 in that article, a generating system for the Types $A$, $D$ and $E$ is given, but no basis seems to be known.

%The idea behind the projections is as follows: We embed $\LOP$ into the space of functions on the symmetric group with the natural left action. Then the subspaces arise as the invariant subspaces.

%This note is organized as follows: In the next section we start with some general considerations. In the third section, we prove the results.

\section{General considerations}
Consider the vector space $\RR^N$ with the usual scalar product $\scp{}{}$ and the standard orthonormal basis $e_1, \ldots, e_N$.
Let $\Delta = \set{\sum_{i} \lambda_{i} e_{i} \with 0\leq \lambda_{i} \leq 1, 1\leq i\leq N}$ denote the standard simplex.

To every subspace $U \subset \RR^N$ we can associate a polytope $\poly{U}$ as follows: Let $\poly{U}$ be the image of $\Delta$ under the orthogonal projection $\RR^N \rightarrow U$. If $V \subset U$ is another subspace, the projection $\RR^N \rightarrow V$ factors through $U$, giving a linear projection $\poly{U} \rightarrow \poly{V}$.

There is a second way to associate a polytope to a subspace $U \subset \RR^N$: Choose a generating system $b_1, \ldots, b_l$ of $U$ and write the vectors as rows into a matrix $\mathcal B$. Then we define $\polyy{b_1, \ldots, b_l}$ to be the convex hull of the column vectors of $\mathcal B$ in the $\RR^l$.

\begin{lemma}\label{lemma:basis}
The polytopes $\poly{U}$ and  $\polyy{b_1, \ldots, b_l}$ are affinely isomorphic.
\end{lemma}
\begin{proof}
The map $\RR^N \rightarrow \RR^l$ defined by the matrix ${\mathcal B}$ maps $\poly{U}$ bijectively onto $\polyy{b_1, \ldots, b_l}$. This can be checked by a short calculation.
\end{proof}

Now we specialize to the situation where a finite group $G$ acts on the set of basis vectors $e_1, \ldots e_N$ of $\RR^N$. We extend this action to a permutation representation of $G$. Note that the scalar product and the simplex $\Delta$ are invariant under this action. We further assume that $U \subset \RR^N$ is an invariant subspace. Then $\poly{U}$ is invariant as well and the orthogonal projection is an equivariant mapping. Therefore, the action of $G$ induces automorphisms of $\poly{U}$.
If $V \subset U$ are both invariant subspaces, then the projection $\poly{U} \rightarrow \poly{V}$ is equivariant.

%----------------------------------------------------------------------------------------------------------------------------

\section{The Linear Ordering Polytope}
Consider the space of functions $\kg = \set{f:\Sn \rightarrow \RR}$. There is a natural left action of the $\Sn$ on this space by $(\pi f)(\tau) \defa f(\tau \pi)$. We choose the canonical basis $e_{\pi}, \pi\in\Sn$ defined by $e_{\pi}(\tau) = \delta_{\pi, \tau}$. This basis is permuted by the action: $\tau e_\pi = e_{\pi\tau^{-1}}$, hence we are in the situation considered at the end of Section 2.
The vector space $\kg$ also carries a natural invariant scalar product $\scp{f}{g} = \sum_{\pi\in\Sn} f(\pi)g(\pi)$.
Let $\eins$ denote the constant function with value $1$ on $\Sn$.

Regard the $k_{ij}$ defined in the introduction as functions in $\kg$. Then the linear ordering polytope is $\polyy{k_{ij}, 1\leq i<j \leq n}$. We compute the action of the $\Sn$ on the $k_{ij}$:
\[\pi k_{ij} =
\begin{cases}
k_{\pi(i)\pi(j)} &\textnormal{ if } \pi(i) > \pi(j) \\
\eins - k_{\pi(j)\pi(i)} &\textnormal{ if } \pi(i) < \pi(j) \\
\end{cases}\]
Therefore, the vector space spanned by the $k_{ij}$ is not invariant, but the space $\Uinv$ generated by the $k_{ij}$ together with $\eins$ is. The polytope $\polyy{\eins, k_{ij}, 1\leq i<j\leq n}$ is clearly affinely isomorphic to $\LOP$ because the $\eins$ only adds a new coordinate with constant value $1$ to the polytope, thus moving it into an affine hyperplane.
Hence we can regard the linear ordering polytope as $\poly{\Uinv}$. 
%The mapping $\kg \rightarrow \RR^{\binom{n}{2}}, e_\pi \mapsto \lop_\pi$ gives rise to an $\Sn$-action relabeling action defined above.
The left action of $\Sn$ is what is called \emph{relabeling} action in \cite{fiorini2001determining}.
There is an additional symmetry on this space: Let $\wn \in S_n$ denote the reversal permutation defined by $\wn(1) = n, \wn(2) = n-1, \ldots, \wn(n) = 1$. It defines a $\mathbb{Z}_2$-action on $\kg$ via $(\wn f)(\tau) \defa f(\wn\tau)$
which commutes which the $S_n$-action mentioned before.
We compute $\wn k_{ij} = \eins - k_{ij}$, hence the space $\Uinv$ is also invariant under $\wn$.
This defines the \emph{duality} automorphism of \cite{fiorini2001determining}. 

Next we decompose $\Uinv$ into invariant subspaces. % The projections in Theorem 2 and Proposition 1 of $P_n$ will then arise as the projections onto those.
First we remove the $1$-dimensional subspace spanned by $\eins$. Denote it by $V_0 \defa \RR \eins$.
A new basis for the complement $\tilde{U}_{Inv}$ is given by
\[\tk_{ij}(\pi) \defa 
\begin{cases}
1 &\textnormal{ if } \pi(i) > \pi(j) \\
-1 &\textnormal{ if } \pi(i) < \pi(j) \\
\end{cases}\]
for $1\leq i < j \leq n$. Thus, we have $\tk_{ij} = 2k_{ij}-\eins$. It is convenient to define also $\tk_{ii}\defa 0$ and $\tk_{ji} \defa - \tk_{ij}$ for $i<j$. The action of $\mathbb{Z}_2 \times \Sn$ is
\begin{align*}
\pi \tk_{ij} &= \tk_{\pi(i)\pi(j)} \\
%\begin{cases}
%\tk_{\pi(i)\pi(j)} &\textnormal{ if } \pi(i) > \pi(j) \\
%- \tk_{\pi(j)\pi(i)} &\textnormal{ if } \pi(i) < \pi(j) \\
%\end{cases} \\
\wn \tk_{ij} &= -\tk_{ij} \,.
\end{align*}
Thus the subspace $\tilde{U}_{Inv}$ spanned by the $\tk_{ij}$ without $\eins$ is indeed an invariant subspace. Using the same argument as above, one sees that $\polyy{\tk_{ij}, 1\leq i < j \leq n}$ is isomorphic to $\LOP$. Since every subspace of $\tilde{U}_{Inv}$ is invariant under the $\wn$-action, we only need consider the $S_n$-action.

To understand the structure of $\tilde{U}_{Inv}$ we consider an additional space:
Let $F$ be an $n$-dimensional vector space with basis $e_1, \ldots, e_n$ and the standard $\Sn$-representation $\pi e_i \defa e_{\pi(i)}$. Then the mapping $\psi: F \wedge F \rightarrow \tilde{U}_{Inv} \, , \, e_i \wedge e_j \mapsto \tk_{ij}$ is an equivariant isomorphism.

It is well-known that $F$ decomposes into a direct sum of two invariant subspaces: $F_0$ with basis $\sum e_i$ and $F_1$ with basis $e_i - e_n, 1\leq i \leq n-1$. A short computation yields the decomposition of $\tilde{U}_{Inv} \cong F\wedge F = F_1 \wedge F_0 \oplus F_1 \wedge F_1$. For both summands it is known that they are irreducible, so we found the decomposition of $\tilde{U}_{Inv}$. Since the two subspaces are irreducible and non-isomorphic as representations, they are orthogonal with respect to the invariant scalar product.
Let $V_1$ and $V_2$ denote the spaces isomorphic to $F_0 \wedge F_1$ reps. $F_1 \wedge F_1$. We consider these spaces separately: 

\subsection{The space $V_1$}
We can get a basis for $V_1$ by applying the isomorphism $\psi$ to any basis of $F_1$. However, to see that $\poly{V_1}$ is a permutahedron, we use a particular generating system:
\[ v_i \defa \sum_{j=1}^n \tk_{ij}\]
It is easy to see that $\tau v_i = v_{\tau(i)}$. This is the same $S_n$-representation as on $V_1$. Since we already know that $\tilde{U}_{Inv}$ has only one invariant subspace with this representation, we conclude that the subspace generated by $v_j, 1\leq j \leq n$ is $V_1$. A short calculation reveals that $v_i(\pi) = 2\pi(i) - (n+1)$. Therefore, $\polyy{v_1,\ldots v_n}$ is affinely isomorphic to a permutahedron.

\subsection{The space $V_2$}
Recall that $e_i - e_n, 1\leq i \leq n-1$ is a basis for $F_1$. Hence $(e_i-e_n)\wedge (e_j-e_n), 1\leq i < j \leq n-1$ is a basis of $F_1 \wedge F_1$. We apply $\psi$ and get the following basis of $V_2$:
\[\tw_{ij} \defa \tk_{ij}-\tk_{in} + \tk_{jn}\]
 for $1 \leq i < j \leq n-1$. Define $\cyc$ to be the cycle $(1 2 \ldots n) \in S_n$ and let $C_n$ be the subgroup generated by $\cyc$. Note that $\tw_{ij}(\cyc \pi) = \tw_{ij}(\pi)$ for every $i<j$ and $\pi\in\Sn$.
This can be checked directly from the definition of the $\tw_{ij}$.
Thus vertices of $\polyy{\tw_{ij}, 1\leq i<j\leq n}$ can be labelled with the right cosets of $C_n$.
For each coset $C_n \pi$ there is a unique representative $\pi' \in \Sn$ with $\pi(n)=n$ and $\tw_{ij}(\pi) = \tw_{ij}(\pi')$. Therefore, the polytope does not change if we restrict to the subspace of permutations fixing $n$. This can be naturally identified with $\RR^{S_{n-1}}$. 
But for $\pi \in\Sn$ with $\pi(n)=n$, we have $\tw_{ij}(\pi) = \tk_{ij}(\pi)$. Hence, the polytope is the $(n-1)$-st linear ordering polytope.

\section*{Acknowledgements}
I wish to thank S. Fiorini for pointing out that the vector space decomposition from Theorem \ref{thm:main} was known in combinatorial optimization.

\bibliography{Literatur}
\bibliographystyle{amsplain}

\end{document}